\documentclass[12pt,reqno]{amsart}
\usepackage{amssymb, a4wide, amsmath,amsthm, mathtools,xy}
\xyoption{all}
\usepackage{latexsym,pdfsync,xcolor,graphicx}
\usepackage{multirow}

\usepackage{array}
\usepackage{upgreek}
\usepackage{pstricks-add}
\usepackage{enumerate}

\usepackage[english]{babel}	
\usepackage{comment}
\allowdisplaybreaks

\usepackage{pgf,tikz}
\usepackage{wrapfig}
\usetikzlibrary{arrows}

\usepackage{float}
\usepackage{appendix}

\usepackage{hyperref}
\usepackage{rotating}
\newtheorem{theorem}{Theorem}

\newtheorem{corollary}[theorem]{Corollary}
\newtheorem{proposition}[theorem]{Proposition}

\newtheorem{lemma}[theorem]{Lemma}

\theoremstyle{definition}

\newtheorem{example}[theorem]{Example}

\theoremstyle{remark}
\newtheorem{remark}[theorem]{Remark}

\DeclareMathOperator{\spa}{span}

\DeclareMathOperator{\supp}{supp}

\title[Regular evolution algebras are closed under subalgebras]{Regular evolution algebras are closed \\ under subalgebras}

\author[M. Ladra]{Manuel Ladra\textsuperscript{1}}
\address{\textsuperscript{1}Department of Mathematics \& CITMAga, Universidade de Santiago de Compostela, 15782 Santiago de Compostela, Spain}
\email{manuel.ladra@usc.es}

\author[A. Pérez-Rodríguez]{Andrés Pérez-Rodríguez\textsuperscript{2}}
\address{\textsuperscript{2}Department of Mathematics \& CITMAga, Universidade de Santiago de Compostela, 15782 Santiago de Compostela, Spain}
\email{andresperez.rodriguez@usc.es}



\begin{document}
\maketitle

\begin{abstract} The main goal of this note is to show that subalgebras of regular evolution algebras are themselves evolution algebras. This allows us to assume, without loss of generality, that every subalgebra in the regular setting has a basis consisting of vectors with disjoint supports. Finally, we use this result to characterise the existence of codimension-one subalgebras in regular evolution algebras.
\end{abstract}
\vspace{0.2cm}

Given an arbitrary field $\mathbb{K}$, an \textit{evolution algebra} over $\mathbb{K}$ is a $\mathbb{K}$-algebra $\mathcal{E}$ which admits a distinguished basis $B=\{e_1,\dots,e_n,\dots\}$, called a \textit{natural basis}, satisfying the condition that $e_ie_j=0$ for all $i\neq j$. In this note, we focus on finite-dimensional evolution algebras, meaning that $B$ is a finite set. For a given natural basis $B=\{e_1,\dots,e_n\}$ in $\mathcal{E}$, the scalars $a_{ij}\in\mathbb{K}$ satisfying $e_i^2=\sum_{j=1}^na_{ij}e_j$ are called the \textit{structure constants} of $\mathcal{E}$ relative to $B$. The matrix $M_B(\mathcal{E})=(a_{ij})_{i,j=1}^n$ is said to be the \textit{structure matrix} of $\mathcal{E}$ relative to $B$. Evolution algebras were introduced by J. P. Tian and P. Vojt\v{e}chovsk\'{y} as a mathematical framework for modelling inheritance mechanisms that do not follow classical Mendelian genetics. Their commutative but non-associative multiplication reflects self-replication patterns observed in specific non-Mendelian genetic models (see \cite{Tian_08,TV_06}). It is worth noting that linear algebra has provided powerful tools for modelling and analysing genetic systems since they were first studied by Etherington \cite{E_40}.

Since evolution algebras are not defined by identities, one of the main challenges in working with them is that they are neither closed under subalgebras (see \cite[Example~2.3]{CSV_16}) nor under ideals (see \cite[Example~2.7]{CSV_16}). However, unlike the case of subalgebras, the structure of ideals has been extensively studied (see \cite{CBGT_24,CKS_19}). In particular, \cite[Proposition~4.2]{BCS_22_nat} establishes that if an evolution algebra $\mathcal{E}$ is \textit{regular} ($\mathcal{E}=\mathcal{E}^2$, or equivalently, $M_B(\mathcal{E})$ is non-singular) then every nonzero ideal is \textit{basic}, that is, every ideal admits a natural basis consisting of vectors from the natural basis.

Regular evolution algebras have been thoroughly investigated and possess particularly desirable properties. For instance, they have a unique natural basis (see \cite[Theorem~4.4]{EL_15}), their automorphism groups are finite (see \cite[Theorem~4.8]{EL_15}), and they are universally finite (see \cite{CLTV_21}). Moreover, their algebras of derivations have also been described \cite[Theorem~4.1]{EL_21}. Indeed, this short paper aims to explore subalgebras in the regular setting.
First, note that the one-dimensional subalgebras are given by the non-trivial solutions of a non-linear polynomial system of equations.
\begin{lemma}
	Let $\mathcal{E}$ be a regular evolution algebra with natural basis $B=\{e_1,\dots,e_n\}$ and structure matrix $M_B(\mathcal{E})$. Then, the one-dimensional subspace $\spa\big\{\alpha_1e_1+\dots+\alpha_ne_n\}$ is a proper nonzero subalgebra of $\mathcal{E}$ if and only if the vector $(\alpha_1,\dots,\alpha_n)$ is a non-trivial solution of 
	\begin{equation}\label{sist}
		\begin{pmatrix}
			x_1^2 \\ \vdots \\ x_n^2
		\end{pmatrix}=\big(M_B(\mathcal{E})^t\big)^{-1}\begin{pmatrix}
			x_1 \\ \vdots \\ x_n
		\end{pmatrix}.
	\end{equation}
\end{lemma}
\begin{proof}
	A subspace $\spa\{\alpha_1e_1+\dots+\alpha_ne_n\}$ is a subalgebra if and only if there exists a scalar $k\neq0$ such that
	\begin{align*}
		\left(\sum_{i=1}^n\alpha_ie_i\right)^2 = \sum_{i=1}^{n} \alpha_i^2 (a_{i1} e_1 + \dots + a_{in} e_n)= \sum_{j=1}^{n} \left( \sum_{i=1}^{n} a_{ij} \alpha_i^2 \right) e_j=k\left(\sum_{i=1}^n\alpha_ie_i\right).
	\end{align*}
	For this condition to hold, it is necessary that $\sum_{i=1}^{n} a_{ij} \alpha_i^2=k\alpha_i$ for all $j=1,\dots,n$. Moreover, since it can be assumed that $k=1$ (by setting $\alpha_i^\prime=\frac{\alpha_i}{k}$), we obtain~\eqref{sist}.
\end{proof}

In particular, our main theorem establishes that regular evolution algebras are closed under subalgebras, a property that provides a more precise and structured approach to understanding subalgebras and significantly simplifies their study in higher dimensions.

\begin{theorem}\label{th:reg_basis}
	Let $\mathcal{E}$ be a regular evolution algebra over any field $\mathbb{K}$. Then, every subalgebra of $\mathcal{E}$ admits a natural basis.
\end{theorem}
\begin{proof}
	Let $\mathcal{E}$ be a regular evolution algebra with natural basis $\{e_1,\dots,e_n\}$, and
	Consider a proper subalgebra  $U=\spa\{u_i\colon 1\leq i\leq m\}$, where each $u_i=\sum_{j=1}^{n}\mu_{ij}e_j$ with $\mu_{ij}\in\mathbb{K}$, of dimension $m<n$. 
	Without loss of generality, we assume that the matrix $(\mu_{ij})_{i,j=1}^m,n$ is in reduced row echelon form. Furthermore, by appropriately reordering the natural basis, we can assume that the leading coefficients form the identity matrix, i.e. 
	\[u_1=e_1+v_1,\; u_2=e_2+v_2,\;\dots,\;u_m=e_m+v_m,\]
	where $v_1,\dots,v_m\in\spa\{e_{m+1},\dots,e_n\}.$
	
	For the sake of contradiction, assume that $U$ does not admit a natural basis. Then, there exist distinct indices $k,l\in\{1,\dots,m\}$ such that $u_ku_l\neq0$. Moreover, as $\mathcal{E}$ is regular, we have $\operatorname{rank}\{e_1^2,\dots,e_n^2\}=n$, which implies that $\operatorname{rank}\{u_1,\dots,u_m\}=\operatorname{rank}\{u_1^2,\dots,u_m^2\}=m$. Then, as $\{u_1^2,\dots,u_m^2\}$ is also a basis of the subalgebra $U$, there exist scalars $\alpha_1,\dots,\alpha_m\in\mathbb{K}$ such that
	\[\alpha_1u_1^2+\dots+\alpha_mu_m^2+u_ku_l=0.\]
	Expanding each term, 
	\[\alpha_1(e_1^2+v_1^2)+\dots+\alpha_m(e_m^2+v_m^2)+u_ku_l=0.\]
	Since $u_ku_l\in\spa\{e_{m+1}^2,\dots,e_n^2\}$, the linear independence of $\{e_1^2,\dots,e_n^2\}$ forces $\alpha_1=\dots=\alpha_m=0$, contradicting that $u_ku_l\neq0$.
	\end{proof}
	
	\begin{remark}
		The converse of the previous result does not hold in general. As shown in \cite[Corollary 3.4]{LPP_25_lattice}, all subalgebras of a nilpotent evolution algebra with the maximum index of nilpotency are basic ideals. For instance, the unique proper nonzero subalgebras of the evolution algebra $\mathcal{E}$ with natural basis $\{e_1,e_2,e_3\}$ and product given by $e_1^2=e_2$, $e_2^2=e_3$ and $e_3^2=0$ are $\spa\{e_3\}$ and $\spa\{e_2,e_3\}$.
	\end{remark}
	
	Recall that, given an evolution algebra $\mathcal{E}$ with natural basis $B=\{e_1,\dots,e_n\}$ and an element $u=\sum_{i=1}^n\mu_ie_i\in\mathcal{E}$, we define its \textit{support} as $\supp_B(u)\coloneqq\{i:\mu_i\neq0\}$.
	\begin{corollary}
		Let $\mathcal{E}$ be a regular evolution algebra over any field $\mathbb{K}$ with natural basis $B=\{e_1,\dots,e_n\}$. Then, every subalgebra $U$ of $\mathcal{E}$ admits a basis $\{u_1,\dots,u_m\}$ where the supports of its elements are disjoint.
	\end{corollary}
	\begin{proof}
		The result follows from Theorem \ref{th:reg_basis} and from the fact that if we have two elements $u$ and $v$ in a regular evolution algebra, then $uv=0$ if and only if their supports are disjoint (see \cite[Remark 4.3]{EL_15}). 
	\end{proof}
	
	This result significantly simplifies the study of subalgebras in the regular case. For example, given a regular evolution algebra $\mathcal{E}$ with natural basis $B=\{e_1,\dots,e_n\}$, any subalgebra of codimension one can be written as
	\begin{align}\label{eq:sub_codim1}
	\spa\{e_i,v\colon i\neq p,q\},
	\end{align}
	where $v\in\spa\{e_p,e_q\}$ for some distinct indices $p,q\in\{1,\dots,n\}$. Without loss of generality, in this work we will assume that $p<q$. Using this description, we finally establish precise conditions for the existence of subalgebras of codimension one in regular evolution algebras.

	\begin{proposition}\label{prop:sub_codim1}
	Let $\mathcal{E}$ be a regular evolution algebra over any field $\mathbb{K}$ of dimension greater than two with natural basis $B=\{e_1,\dots,e_n\}$ and structure matrix $M_B(\mathcal{E})=(a_{ij})_{i,j=1}^n$. 
	
	Suppose that $\mathcal{E}$ has a subalgebra of codimension one. In this case, there necessarily exist distinct indices $p,q\in\{1,\dots,n\}$ such that \[\operatorname{rank}M_{p,q}\leq1,\]where $M_{p,q}$ denotes the submatrix of $M_B(\mathcal{E})$ formed by the $p$-th and $q$-th columns after removing the $p$-th and $q$-th rows.
	
	Conversely, if $\operatorname{rank}M_{p,q}\leq1$, the existence of subalgebras of codimension one can be characterised as follows:
	\begin{enumerate}[\rm (a)]
	\item If $\operatorname{rank}M_{p,q}=1$, then at least one row of $M_{p,q}$ is non-zero; let $(\alpha,\beta)$ be such a row. In this 
	case, $\mathcal{E}$ has a subalgebra of codimension one if and only if
	\begin{align}\label{eq:cond}
		\alpha^2\beta a_{pp}+\beta^3a_{qp}=\alpha^3a_{pq}+\alpha\beta^2a_{qq},
	\end{align}
	and the corresponding subalgebra is given by \eqref{eq:sub_codim1} with $v=\alpha e_p+\beta e_q$.
	
	\item  If $\operatorname{rank}M_{p,q}=0$, then the subalgebras of codimension one are \eqref{eq:sub_codim1}
	where $v=e_p+\lambda e_q$ for any $\lambda\in\mathbb{K}^*$ satisfying the equation
	\begin{align}\label{eq:cond_pol}
		a_{qp}\lambda^3-a_{qq}\lambda^2+ a_{pp}\lambda-a_{pq}=0,
	\end{align} and additionally, $\spa\{e_i\colon i\neq q\}$ when $a_{pq}=0$ and $\spa\{e_i\colon i\neq p\}$ when $a_{qp}=0$.
	\end{enumerate}
	\end{proposition}
\begin{proof}
Let $\pi_{p,q}$ be the linear projection of $\mathcal{E}$ onto $\spa\{e_p,e_q\}$. Observe that \eqref{eq:sub_codim1} with $v=\alpha e_p+\beta e_q$ defines a subalgebra of $\mathcal{E}$ if and only if the following conditions hold: 
\begin{align}
\pi_{p,q}(e_i^2)&=a_{ip}e_p+a_{iq}e_q\in\spa\{\alpha e_p+\beta e_q\}, \text{ for all }i\neq p,q;\text{ and } \label{eq:cond1} \\
\pi_{p,q}\big((\alpha e_p+\beta e_q)^2\big)&=(\alpha^2a_{pp}+\beta^2a_{qp})e_p+(\alpha^2a_{pq}+\beta^2a_{qq})e_q\in\spa\{\alpha e_p+\beta e_q\}.\label{eq:cond2}
\end{align}
From \eqref{eq:cond1}, it follows that $\operatorname{rank}M_{p,q}\leq1$ is a necessary condition.  
Conversely, now assume that $\operatorname{rank}{M_{p,q}}\leq1$ and study the following two cases separately:
\begin{enumerate}[\rm (a)]
\item If $\operatorname{rank}{M_{p,q}}=1$, then $(\alpha,\beta)$ must correspond with a non-zero row of $M_{p,q}$ for \eqref{eq:sub_codim1} to be a subalgebra of $\mathcal{E}$.
Now, we consider the following scenarios for \eqref{eq:cond2} to be satisfied:
\begin{enumerate}[\rm (i)]
\item if $\alpha=0$ and $\beta\neq0$, then \eqref{eq:cond2} holds if and only if $a_{qp}=0$;
\item if $\beta=0$ and $\alpha\neq0$, then \eqref{eq:cond2} holds if and only if $a_{pq}=0$;
\item if $\alpha,\beta\neq0$, then \eqref{eq:cond2} holds if and only if there exists a scalar $k\in\mathbb{K}^*$ such that $\alpha^2a_{pp}+\beta^2a_{qp}=k\alpha$ and $\alpha^2a_{pq}+\beta^2a_{qq}=k\beta$, or equivalently, if \eqref{eq:cond} holds.
\end{enumerate}
Notice that if $\alpha=0$ and $\beta\neq0$ (resp. $\beta=0$ and $\alpha\neq0$), then \eqref{eq:cond} holds if and only if $a_{qp}=0$ (resp. $a_{pq}=0$). Thus, we conclude that, in this case, \eqref{eq:sub_codim1} is a subalgebra if and only if \eqref{eq:cond}.

\item If $\operatorname{rank}{M_{p,q}}=0$, then it is evident that $\spa\{e_i\colon i\neq q\}$ and $\spa\{e_i\colon i\neq p\}$ are subalgebras if and only if $a_{pq}=0$ and $a_{qp}=0$, respectively. Moreover, \eqref{eq:sub_codim1} with $v=\alpha e_p+\beta e_q$ for some $\alpha,\beta\neq0$, or equivalently, with  $v=e_p+\lambda e_q$ for $\lambda\neq0$, is a subalgebra if and only if there exists a scalar $k\in\mathbb{K}^*$ such that $a_{pp}+\lambda^2 a_{qp}=k$ and $a_{pq}+\lambda^2 a_{qq}=k\lambda$, which is equivalent to $\lambda$ being a nonzero solution of \eqref{eq:cond_pol}.
\end{enumerate} 
Thus, the result follows.
\end{proof}
\begin{remark}
	Although every nonzero proper subalgebra of a regular evolution algebra of dimension two trivially admits a natural basis, as they are of dimension one, the computations in Proposition \ref{prop:sub_codim1} also play a key role in fully characterising all subalgebras in this case. Specifically, given a regular evolution algebra $\mathcal{E}$ over any field $\mathbb{K}$ with natural basis $\{e_1,e_2\}$ and structure matrix $M_B(\mathcal{E})=(a_{ij})_{i,j=1}^2$, the subalgebras are given by $\spa\{e_1+\lambda e_2\}$
	for any $\lambda\in\mathbb{K}^*$ satisfying the equation
	\begin{align*}
		a_{21}\lambda^3-a_{22}\lambda^2+ a_{11}\lambda-a_{12}=0;
	\end{align*} and additionally, $\spa\{e_1\}$ when $a_{12}=0$ and $\spa\{e_2\}$ when $a_{21}=0$.
\end{remark}
\begin{corollary}
	Let $\mathcal{E}$ be a regular evolution algebra over any field $\mathbb{K}$ of dimension greater than two,  with natural basis $B=\{e_1,\dots,e_n\}$ and structure matrix $M_B(\mathcal{E})=(a_{ij})_{i,j=1}^n$. If the subspace \eqref{eq:sub_codim1} is a subalgebra for some $v\in\spa\{e_p,e_q\}$, then the structure constants satisfy $$a_{ip}^2a_{iq}a_{pp}+a_{iq}^3a_{qp}=a_{ip}^3a_{pq}+a_{ip}a_{iq}^2a_{qq}$$ for all $i\neq p,q$. Particularly, if $\mathcal{E}$ is three-dimensional and $\mathbb{K}$ is algebraically closed or $\mathbb{K}=\mathbb{R}$, then the converse also holds.
	\end{corollary}
\begin{proof}
	The necessity follows straightforwardly from Proposition \ref{prop:sub_codim1}.
	
	For sufficiency in dimension three, first note we always have that $\operatorname{rank}M_{p,q}\leq1$ for any distinct indices $p,q\in\{1,2,3\}$. Moreover, as $\mathbb{K}$ is algebraically closed or $\mathbb{K}=\mathbb{R}$, then \eqref{eq:cond_pol} with $a_{pq},a_{qp}\neq0$ always has a solution in $\mathbb{K}$. 
\end{proof}
We conclude this note with two examples which show how the hypothesis for equivalence cannot be relaxed. Example \ref{ex:1} illustrates the importance of the choice of the field, while Example \ref{ex:2} highlights the necessity of the evolution algebra to have dimension three.
\begin{example}\label{ex:1}
	Let $\mathcal{E}$ be the regular evolution algebra over $\mathbb{Q}$ with natural basis $\{e_1,e_2,e_3\}$ and structure matrix given by
	\[M_B(\mathcal{E})=\begin{pmatrix*}[r]
		1&0&0 \\ 1&-1&1 \\ 2&1&0
	\end{pmatrix*}.\]
	Moreover, the structure constants satisfy the relation $$a_{12}^2a_{13}a_{22}+a_{13}^3a_{32}=a_{12}^3a_{23}+a_{12}a_{13}^2a_{33}.$$ However, according to Proposition \ref{prop:sub_codim1}, we have that $a_{23},a_{32}\neq0$, and the polynomial $$a_{32}x^3-a_{33}x^2+ a_{22}x-a_{23}=x^3-x-1$$ is irreducible over $\mathbb{Q}$. Consequently, there are no subalgebras of dimension two of the form $\spa\{e_1,v\}$ with $v\in\spa\{e_2,e_3\}$. Moreover, since$$a_{21}^2a_{23}a_{11}+a_{23}^3a_{31}\neq a_{21}^3a_{13}+a_{21}a_{23}^2a_{33}\quad\text{and}\quad a_{31}^2a_{32}a_{11}+a_{32}^3a_{21}\neq a_{31}^3a_{12}+a_{31}a_{32}^2a_{22},$$ $\mathcal{E}$ does not have subalgebras of dimension two.  
\end{example}
\begin{example}\label{ex:2}
Let $\mathcal{E}$ be a regular evolution algebra over any field $\mathbb{K}$ with natural basis $\{e_1,e_2,e_3,e_4\}$ and the two last columns of the structure matrix $M_B(\mathcal{E})$ given by
	\[\begin{pmatrix*}[r]
			a_{13}&a_{14} \\ a_{23}&a_{24} \\ a_{33}&a_{34} \\ a_{43}&a_{44}
		\end{pmatrix*}=\begin{pmatrix*}[r]
		1&2 \\ 1&-1 \\ -3&2 \\1&0
	\end{pmatrix*}.\]
Notice that $a_{i3}^2a_{i4}a_{33}+a_{i4}^3a_{43}=a_{i3}^3a_{34}+a_{i3}a_{i4}^2a_{44}$ for $i=1,2$. However, $$\operatorname{rank}M_{3,4}=\operatorname{rank}\begin{pmatrix*}[r]
		1&2 \\ 1&-1
	\end{pmatrix*}=2,$$ what yields, by Proposition \ref{prop:sub_codim1}, that there are no subalgebras of dimension three of the form $\spa\{e_1,e_2,v\}$ with $v\in\spa\{e_3,e_4\}$.
\end{example}

\section*{Acknowledgements}
This work was partially supported by Agencia Estatal de Investigaci\'on (Spain),
grant PID2020-115155GB-I00 (European FEDER support included, UE) and
by Xunta de Galicia through the Competitive Reference Groups (GRC), ED431C
2023/31.
The third author was also supported by FPU21/05685 scholarship, Ministerio de Educaci\'on y Formaci\'on Profesional (Spain).


\begin{thebibliography}{1}
\expandafter\ifx\csname url\endcsname\relax
  \def\url#1{\texttt{#1}}\fi
\expandafter\ifx\csname urlprefix\endcsname\relax\def\urlprefix{URL }\fi

\bibitem{BCS_22_nat}
N.~Boudi, Y.~Cabrera~Casado, M.~Siles~Molina, Natural families in evolution
  algebras, Publ. Mat. 66~(1) (2022) 159--181.
\newline\urlprefix\url{https://doi.org/10.5565/publmat6612206}

\bibitem{CBGT_24}
Y.~Cabrera~Casado, D.~M. Barquero, C.~M. Gonz{\'a}lez, A.~Tocino, Connecting
  ideals in evolution algebras with hereditary subsets of its associated graph,
  Collect. Math. (2024) 1--16.
\newline\urlprefix\url{https://doi.org/10.1007/s13348-024-00435-x}

\bibitem{CKS_19}
Y.~Cabrera~Casado, M.~Kanuni, M.~Siles~Molina, Basic ideals in evolution
  algebras, Linear Algebra Appl. 570 (2019) 148--180.
\newline\urlprefix\url{https://doi.org/10.1016/j.laa.2019.01.010}

\bibitem{CSV_16}
Y.~Cabrera~Casado, M.~Siles~Molina, M.~V. Velasco, Evolution algebras of
  arbitrary dimension and their decompositions, Linear Algebra Appl. 495 (2016)
  122--162.
\newline\urlprefix\url{https://doi.org/10.1016/j.laa.2016.01.007}

\bibitem{CLTV_21}
C.~Costoya, P.~Ligouras, A.~Tocino, A.~Viruel, Regular evolution algebras are
  universally finite, Proc. Amer. Math. Soc. 150~(3) (2022) 919--925.
\newline\urlprefix\url{https://doi.org/10.1090/proc/15648}

\bibitem{EL_15}
A.~Elduque, A.~Labra, Evolution algebras and graphs, J. Algebra Appl. 14~(7)
  (2015) 1550103, 10 pp.
\newline\urlprefix\url{https://doi.org/10.1142/S0219498815501030}

\bibitem{EL_21}
A. Elduque, A. Labra, Evolution algebras, automorphisms, and graphs, Linear Multilinear Algebra
69 (2) (2021) 331--342.
\newline\urlprefix\url{https://doi.org/10.1080/03081087.2019.1598931}

\bibitem{E_40}
 I.~M.~H.~Etherington, Genetic algebras, Proc. Roy. Soc. Edinburgh 59 (1939) 242--258.
\newline\urlprefix\url{https://doi.org/10.1017/S0370164600012323}

\bibitem{LPP_25_lattice}
M.~Ladra, P.~Páez-Guillán, A.~Pérez-Rodríguez, On the subalgebra lattice of
  solvable evolution algebras.
\newline\urlprefix\url{https://arxiv.org/abs/2502.05619}

\bibitem{Tian_08}
J.~P. Tian, Evolution algebras and their applications, vol. 1921 of Lecture
  Notes in Mathematics, Springer, Berlin, 2008.
\newline\urlprefix\url{https://doi.org/10.1007/978-3-540-74284-5}

\bibitem{TV_06}
J.~P. Tian, P.~Vojt{\v{e}}chovsk{\'y}, Mathematical concepts of evolution
  algebras in non-{M}endelian genetics, Quasigroups Related Systems 14~(1)
  (2006) 111--122.

\end{thebibliography}
\end{document}